\documentclass[english]{extarticle}
\usepackage[T1]{fontenc}
\usepackage[latin9]{inputenc}
\usepackage{geometry}
\usepackage{graphicx}
\usepackage{color}
\usepackage{babel}
\usepackage{amsmath}
\usepackage{amsthm}
\usepackage{amssymb}
\usepackage[all]{xy}
\usepackage[unicode=true,pdfusetitle,
 bookmarks=true,bookmarksnumbered=true,bookmarksopen=false,
 breaklinks=false,pdfborder={0 0 1},backref=false,colorlinks=true]
 {hyperref}
\hypersetup{
 linkcolor=blue,urlcolor=blue,citecolor=blue,pdfstartview={FitH},hyperfootnotes=false}

\makeatletter
\theoremstyle{plain}
\newtheorem{thm}{\protect\theoremname}
\theoremstyle{plain}
\newtheorem{cor}[thm]{\protect\corollaryname}
\theoremstyle{plain}
\newtheorem{prop}[thm]{\protect\propositionname}
\theoremstyle{remark}
\newtheorem{rem}[thm]{\protect\remarkname}

\usepackage{lettrine} 
\let\myFoot\footnote
\renewcommand{\footnote}[1]{\myFoot{#1\vspace{3mm}}}
\usepackage{mlmodern}
\usepackage[T1]{fontenc}
\usepackage{amsmath}

\usepackage{tikz}
\usepackage{circuitikz}

\makeatother

\providecommand{\corollaryname}{Corollary}
\providecommand{\propositionname}{Proposition}
\providecommand{\remarkname}{Remark}
\providecommand{\theoremname}{Theorem}

\begin{document}
\title{A note on a deterministic property to obtain the long run behavior of the range of a stochastic process}
\author{Maher Boudabra\footnote{Mathematics department, King Fahd University for Petroleum and Minerals, KSA} , Binghao Wu\footnote{School of Mathematics, Monash University, Australia}}
\maketitle
\begin{abstract}
A Brownian motion with drift is simply a process $V^{\eta}_t$ of the
form $V^{\eta}_t=B_{t}+\eta t$ where $B_{t}$ is a standard Brownian
motion and $\eta>0$ \footnote{The case $\eta<0$ is deducible by remarking $V^{-\eta}(t)=-V^{\eta}(t)$.}
In \cite{tanre2006range}, the authors considered the drifted
Brownian motion and studied the statistics of some related sequences
defined by certain stopping times. In particular, they provided the
law of the range $R_{t}(V^{\eta})$ of $V^{\eta}$ as well as its
first range process $\theta_{V^{\eta}}(a)$. In particular, they investigated
the asymptotic comportment of $R_{t}(V^{\eta})$ and $\theta_{V^{\eta}}(a)$.
They proved that if $V_{t}^{\eta}$ is a Brownian motion with a positive
drift $\eta$ then its range $R_{t}(V^{\eta})=\sup_{0\leq s\leq t}V_{t}^{\eta}-\inf_{0\leq s\leq t}V_{t}^{\eta}$
is asymptotically equivalent to $\eta t$. In other words 

\begin{equation}
\frac{R_{t}(V^{\eta})}{t}\overset{a.e}{\underset{t\rightarrow\infty}{\longrightarrow}}\eta.\label{range}
\end{equation}
In this paper, we show that (\ref{range}) follows from a striking
deterministic property. More precisely, we show that the long run
behavior of the range of a deterministic continuous function is obtainable
straightaway from that of the function itself. Our result can be deemed
as the continuous version of a similar one appeared in \cite{mgrw}. 
\end{abstract}

\section{Introduction}

The range of a real valued process $X=(X_{t})_{t\in T}$ is defined
to be the "measure" (in a certain sense) of the
set of values generated by this process up to time $t$. We shall
denote the range of the process $X$ by $R_{t}(X)$. When $X$ is
discrete, i.e $X=(X_{n})_{n\in\mathbb{N}}$, $R_{n}(X)$ simplifies
to the number of explored sites up to time $n$ by the process $X$,
i.e $R_{n}(X)=\vert\{X_{0},...,X_{n}\}\vert$. The range of discrete
processes (random walks, Markov chains etc...) has been investigated
in several seminal works \cite{feller1951asymptotic,glynn1985range,chong2000ruin,vallois1996range,vallois1997range}.
In particular, we cite Kesten-Spitzer-Whitman theorem \cite{spitzer2001principles}
which says that the average of the range of a nearest neighbor simple
random walk (referred to as its speed in \cite{mgrw}) converges to
the probability of never returning to the origin (the starting point
more generally). That is, let $(\xi_{n})_{n}$ be a sequence of i.i.d
Redmacher random variables of parameter $p$ and set $X_{n}=\xi_{1}+...+\xi_{n}$.
Then 
\begin{equation}
\frac{R_{n}(X)}{n}\overset{a.e}{\underset{n\rightarrow\infty}{\longrightarrow}}\mathbb{P}(X\,\,\text{never returns to the origin}).\label{range X_n}
\end{equation}
Note that $\mathbb{P}(X\,\,\text{never returns to the origin})$ is simply
$2p-1=\vert\mathbb{E}(X_{1})\vert$. In \cite{mgrw}, the authors
showed that (\ref{range X_n}) is a deterministic property, and no
requirement to involve randomness to get it. In fact they showed the
following result.
\begin{thm}
\label{the mgrw} \cite{mgrw} if $x=(x_{n})_{n}$ is any deterministic
sequence mimicking a nearest neighbor random walk, i.e $\vert x_{n+1}-x_{n}\vert\leq1$
for all $n$, then if 
\[
\frac{x_{n}}{n}\underset{{\scriptscriptstyle n\rightarrow+\infty}}{\longrightarrow}\ell
\]
then 
\[
\frac{R_{n}(x)}{n}\underset{{\scriptscriptstyle n\rightarrow+\infty}}{\longrightarrow}\vert\ell\vert.
\]
\end{thm}

Theorem \eqref{the mgrw} shows that the asymptotic behavior of the
range of any nearest neighbor random walk follows immediately from
the asymptotic behavior of the random walk itself. This deterministic
result covers also the case of random walks in random environments \cite{solomon1975random}. An immediate consequence of theorem \eqref{the mgrw} is that
if $X$ is a random walk for which the Birkoff ergodic theorem \cite{shiryaev2016probability} applies
then 
\begin{equation}
\frac{R_{n}(X)}{n}\underset{{\scriptscriptstyle n\rightarrow+\infty}}{\longrightarrow}\vert\mathbb{E}(X_{1})\vert.\label{range limit}
\end{equation}
\\

In continuous time, the range of a process $X=(X_{t})_{t\geq0}$ is
defined similarly. It is the Lebesgue measure of the interval $[\sup_{0\leq s\leq t}X_{s}-\inf_{0\leq s\leq t}X_{s}]$.
In other words
\begin{equation}
R_{t}(X)=\sup_{0\leq s\leq t}X_{s}-\inf_{0\leq s\leq t}X_{s}.\label{range of X_t}
\end{equation}
One should notice that the path of $R_{t}(X)$ is non-decreasing.
Thus, we can define its right continuous inverse 
\[
\theta_{X}(a):=\inf\{t\geq0\mid R_{t}(X)>a\}.
\]
 Note that $R_{t}(X)$ and $\theta_{X}(a)$ share the same monotonicity.
In addition, the following two relations encodes the duality in between
them:
\begin{itemize}
\item $R_{t}(X)<a\Longleftrightarrow\theta_{X}(a)<t.$
\item $\theta_{X}(R_{t}(X))\geq t.$
\end{itemize}
A typical process to study its range in continuous time is obviously
the Brownian motion as well as its derivatives \cite{tanre2006range,imhof1985range,vallois1995decomposing}.
The density of the range of a standard Brownian motion $B=B_t$ is computed in \cite{feller1951asymptotic} for a fixed
time $t$ while the Laplace transform of $\theta_{B}(a)$ is given
in \cite{imhof1985range,vallois1995decomposing}. In \cite{tanre2006range}, the authors considered
the case of a positive drifted Brownian motion $V_{t}^{\eta}:=B_{t}+\eta t$
where  $\eta>0$. In section
$4,$ they provided the distribution functions of $R_{t}(V^{\eta})$ and
$\theta_{V^{\eta}}(a)$. In addition, they proved two limit theorems
for $\theta_{V^{\eta}}(a)$ that can be seen as a law of large numbers
and a central limit theorem. What matters for us in this work is their
result about the long run behavior of the range. They showed that
the range of $V_{t}^{\eta}$ is asymptotically equivalent to $\eta t$
(a.e), or equivalently
\[
\frac{R_{t}(V^{\eta})}{t}\overset{a.e}{\underset{t\rightarrow\infty}{\longrightarrow}}\eta.
\]
 However, their proof is quite technical and it is based on the study
of the right continuous inverse process $\theta_{V^{\eta}}(a):=\inf\{t\geq0\mid R_{t}(V^{\eta})>a\}$,
combined with the use of the subadditive ergodic theorem. Their approach
will be discussed in the sequel. What we have discovered is that the
comportment of the range of $V_{t}^{\eta}$ at infinity does not require
that whole machinery to obtain. In fact, it is a deterministic property
that follows directly from the behavior of the underlying process.
More generally, we show that the long run behavior of the range of
a function $f:[0,\infty)\rightarrow\mathbb{R}$ is deducible from
the long run behavior of the function itself. To this end, recall
that the range of $f$ is defined exactly as in (\ref{range of X_t}),
i.e 
\[
R_{t}(f)=\sup_{0\leq s\leq t}f_{s}-\inf_{0\leq s\leq t}f_{s}.
\]
When $f$ is continuous then its range path is also continuous. Moreover,
when $f$ s monotonic then its range increases over the time. Finally
note that 
\begin{equation}
R_{t}(f)=R_{t}(-f).\label{range under symmetry}
\end{equation}
\\
\section{Results and applications}
Now, we are ready to state the main result of the paper which is inspired from the work in \cite{mgrw} but has more consequences.
\begin{thm}
\label{Thm} Let $f,\psi:[0,\infty)\rightarrow\mathbb{R}$ be two
functions such that $\psi(t)$ is positively increasing to infinity.
If 
\begin{equation}
\frac{f(t)}{\psi(t)}\underset{t\rightarrow\infty}{\longrightarrow}\ell\not = 0\label{f/psi}
\end{equation}
then 
\[
\frac{R_{t}(f)}{\psi(t)}\underset{t\rightarrow\infty}{\longrightarrow}\vert\ell\vert.
\]
\end{thm}

\begin{proof}
Without loss of generality we may assume $\ell$ is positive thanks
to (\ref{range under symmetry}). Let $\varepsilon\in(0,\ell)$. The
assumption (\ref{f/psi}) yields 
\begin{equation}
\psi(t)(\ell-\varepsilon)\leq f(t)\leq\psi(t)(\ell+\varepsilon)\label{R_t}
\end{equation}
for $t$ greater than some positive $T$. As $\psi(t)$ is positive
then $\psi(t)(\ell-\varepsilon)\leq R_{t}(f)$. On the other hand,
we have 
\begin{equation}
R_{t}(f)\leq R_{T}(f)+\psi(t)(\ell+\varepsilon)\label{R_T}
\end{equation}
since $f(t)>0$ for $t>T$. By combining (\ref{R_t}) and (\ref{R_T})
we get 
\[
-2\varepsilon\leq\frac{R_{t}(f)}{\psi(t)}-\ell\leq\frac{R_{T}(f)}{\psi(t)}+\varepsilon.
\]
Since $\frac{R_{T}(f)}{\psi(t)}$ will vanish at infinity, we can
find $T^{\sharp}\geq T$ such that 
\[
-2\varepsilon\leq\frac{R_{t}(f)}{\psi(t)}-\ell\leq2\varepsilon
\]
for all $t$ greater than $T^{\sharp}$ which completes the proof
by taking into account the case $\ell<0$. 
\end{proof}
Note that theorem (\ref{Thm}) covers also the case when $\ell$ is
infinite. This is done by a minor change in the previous proof. Using
the same assumptions and notations of theorem (\ref{Thm}), the case
when $\ell$ is zero is similar and it is the subject of the next
theorem. 
\begin{thm}
\label{Thm-1} If 
\begin{equation}
\frac{f(t)}{\psi(t)}\underset{t\rightarrow\infty}{\longrightarrow}0\label{f/psi-1}
\end{equation}
then 
\[
\frac{R_{t}(f)}{\psi(t)}\underset{t\rightarrow\infty}{\longrightarrow}0.
\]
\end{thm}

\begin{proof}
The assumption (\ref{f/psi-1}) yields 
\begin{equation}
-\psi(t)\varepsilon\leq f(t)\leq\psi(t)\varepsilon\label{R_t-1}
\end{equation}
for $t$ greater than some positive $T$. In particular we get 
\begin{equation}
R_{t}(f)\leq R_{T}(f)+2\psi(t)\varepsilon\label{R_T-1}
\end{equation}
The inequality (\ref{R_T-1}) yields 
\[
\frac{R_{t}(f)}{\psi(t)}\leq\frac{R_{T}(f)}{\psi(t)}+2\varepsilon.
\]
Since $\frac{R_{T}(f)}{\psi(t)}$ will vanish at infinity, we can
find $T^{\dagger}\geq T$ such that 
\[
0\leq\frac{R_{t}(f)}{\psi(t)}\leq3\varepsilon
\]
for all $t$ greater than $T^{\dagger}$ which completes the proof. 
\end{proof}
The long run behavior of the the range of the drifted Brownian motion
$V^{\eta}_t$ follows easily by looking at the behavior of $V^{\eta}_t$
itself followed by applying theorems (\ref{Thm}) and (\ref{Thm-1}). 
\begin{cor}
\label{Brownian drift} If $V_{t}^{\eta}$ is a Brownian motion with
a drift $\eta$ (possibly $\eta=0$) then 
\[
\frac{R_{t}(V^{\eta})}{t}\overset{a.e}{\underset{t\rightarrow\infty}{\longrightarrow}}\vert\eta\vert.
\]
\end{cor}

\begin{proof}
Recall that $V_{t}^{\eta}=B_{t}+\eta t$ where $B_{t}$ is a standard
Brownian motion. Then by the law of large numbers for Brownian
motion (See \cite{morters2010brownian} for example), $B_{t}$ is almost surely
negligible against $t$, that is $\frac{B_{t}}{t}\overset{a.e}{\underset{t\rightarrow\infty}{\longrightarrow}}0$.
Thus $\frac{V_{t}^{\eta}}{t}\overset{a.e}{\underset{t\rightarrow\infty}{\longrightarrow}}\vert\eta\vert$
. Therefore by applying theorems (\ref{Thm}) and (\ref{Thm-1})
we get 
\begin{equation}
\frac{R_{t}(V^{\eta})}{t}\overset{a.e}{\underset{t\rightarrow\infty}{\longrightarrow}}\vert\eta\vert\,\,\,\,\,\,\label{range eta}
\end{equation}
\end{proof}
In particular we obtain the behavior of the range of $B_{t}$ 
\[
\frac{R_{t}(B)}{t}\overset{a.e}{\underset{t\rightarrow\infty}{\longrightarrow}}0
\]
which is not covered in \cite{tanre2006range}. The reason we think
is because the authors derived the limit (\ref{range eta}) by proving
fist the following limit of $\theta_{V^{\eta}}(a)$
\[
\frac{\theta_{V^{\eta}}(a)}{a}\overset{a.e}{\underset{a\rightarrow\infty}{\longrightarrow}}\frac{1}{\vert\eta\vert}
\]
which itself requires computations in terms of $\frac{1}{\eta}$ (valid
only when $\eta\neq0$). Then they deduced the asymptotic behavior
of the range $R_{t}(V^{\eta})$ leaning on the fact that $\theta_{V^{\eta}}(a)$
is the right continuous process of $R_{t}(V^{\eta})$. Now by theorems
(\ref{Thm}) and (\ref{Thm-1}), the asymptotic comportment of the
range of $V^{\eta}$ are obtained independently of $\theta_{V^{\eta}}(x)$
and follows exclusively from the behavior of the underlying process
$V^{\eta}_t$. 

Now we show that the behavior of $\theta_{V^{\eta}}(a)$ is obtainable
from that of $R_{t}(V^{\eta})$ and vice versa. The main tool to do
so is the following result relating the asymptotic behaviors of a
non decreasing function and its generalized inverse. For a survey
about generalized inverses, we refer the reader to \cite{de2015study}.
\begin{prop}
\label{lemma} Let $\xi:\mathbb{R}_{+}\rightarrow\mathbb{R}$ be a
non decreasing function and let $\xi^{\dagger}$ be its generalized
inverse, i.e 
\begin{equation}
\xi^{\dagger}(y)=\inf\{x\mid\xi(x)>y\}.\label{genearlized inverse}
\end{equation}
If 
\[
\frac{\xi(x)}{x}\underset{x\rightarrow\infty}{\longrightarrow}\ell
\]
then 
\[
\frac{\xi^{\dagger}(x)}{x}\underset{x\rightarrow\infty}{\longrightarrow}\frac{1}{\ell}.
\]
\end{prop}

\begin{proof}
The statement is straightforward when $\xi$ is continuously increasing
since in such a case $\xi^{\dagger}(y)=\xi^{-1}(y)$. Otherwise fix
a positive $\varepsilon$ (less than $\ell$ if $\ell>0$) and let
$\tau$ be a large number so that 
\[
x(\ell-\varepsilon)\leq\xi(x)\leq x(\ell+\varepsilon)
\]
for $x\geq\tau$. It is not hard to see that for $x>\tau$, the curve
of $\zeta$ is squeezed between $\frac{x}{\ell+\varepsilon}$ and
$\frac{x}{\ell-\varepsilon}$ \footnote{Jumps (resp. flat parts) of $\xi$ becomes flat parts (resp. jumps)
for $\zeta$}. In other words 
\begin{equation}
\begin{aligned} & \frac{x}{(\ell+\varepsilon)}\leq\xi^{\dagger}(x)\leq\frac{x}{(\ell-\varepsilon)} & \text{if \ensuremath{\ell>0}}\\
 & -\frac{x}{\varepsilon}\leq\xi^{\dagger}(x)\leq\frac{x}{\varepsilon} & \text{if \ensuremath{\ell=0}}
\end{aligned}
\label{generalized}
\end{equation}
for $x$ larger than $\tau$. The result follows by considering the
reciprocals in (\ref{generalized}). 
\end{proof}
\begin{rem}
Another common definition of the generalized inverse is $\xi^{*}(y):=\inf\{x\mid\xi(x)\geq y\}$.
proposition (\ref{lemma}) is still valid if we use $\xi^{*}$ since $\xi^{*}=\xi^{\dagger}$
a.e \footnote{It is because the set of discontinuity of a non decreasing function
is countable}. The only difference is that the formula (\ref{genearlized inverse})
generates a right continuous function while $\xi^{*}$ generates a
left continuous one. A very known usage of the left continuous inverse
is when $\xi$ represents the cumulative distribution function of some distribution
$\mu.$ In such a case, $\xi^{*}$ is used to simulate a random variable
of the same law $\mu$. More precisely 
\[
\xi^{*}(\text{Uni}(0,1))\sim\mu.
\]
That is, the uniform distribution over $(0,1)$ is mapped to a random variable of distribution $\mu$. Note also that proposition (\ref{lemma}) works in both directions since
$\left(\xi^{\dagger}\right)^{\dagger}=\xi$ a.e. In particular, $\ell$ could be infinite. 
\end{rem}

A straightforward consequence of proposition (\ref{lemma}) is the equivalence
between the comportment of $\theta_{V^{\eta}}(x)$ and $R_{t}(V^{\eta})$.
This appeared only as implication in \cite{tanre2006range}. 
\begin{prop}
We have 
\[
\frac{\theta_{V^{\eta}}(x)}{x}\overset{a.e}{\underset{x\rightarrow\infty}{\longrightarrow}}\frac{1}{\eta}.
\]
In particular

\[
\frac{\theta_{B}(x)}{x}\overset{a.e}{\underset{x\rightarrow\infty}{\longrightarrow}}+\infty.
\]
\end{prop}
\newpage
The following diagram illustrates the difference between our approach
and that in \cite{tanre2006range}.

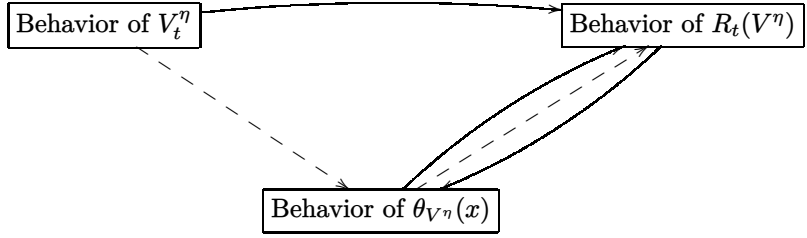
\begin{figure}[h]
\[
\xymatrix{*+[F-]{\text{Behavior of \ensuremath{V_{t}^{\eta}}}}\ar@{-->}[ddr]\ar@/^{0.7pc}/[rr] &  & *+[F-]{\text{Behavior of \ensuremath{R_{t}(V^{\eta})}}}\ar@/^{0.7pc}/[ddl]\\
\\
 & *+[F-]{\text{Behavior of \ensuremath{\theta_{V^{\eta}}(x)}}}\ar@/^{0.7pc}/[uur]\ar@{-->}[uur]
}
\]

\caption{The continuous bent arrows illustrate our approach while the dashed
ones illustrate that in \cite{tanre2006range}.}
\end{figure}

Another application of proposition (\ref{lemma}) is the recovery
of the limit theorem for renewal processes (We refer the reader to
\cite{mitov2014renewal} for more details). That is, let $N_{t}$
be a renewal counting process associated to some renewal sequence $T_{n}$
defined by 
\begin{equation}
N_{t}=\sup\{n\mid T_{n}\leq t\}.\label{renewal}
\end{equation}
Note that the RHS of (\ref{renewal}) is simply $\sum_{n=1}^{\infty}\mathbf{1}_{\{T_{n}\leq t\}}$.
The process $N_{t}$ is the generalized inverse of the renewal sequence
$T_{n}$ (the right continuous one). More precisely, $T_{n}$ is obtainable
from $N_{t}$ by 
\[
T_{n}=\inf\{t\mid N_{t}\geq n\}.
\]
That is, we recover the following equivalence 
\[
\frac{T_{n}}{n}\overset{a.e}{\underset{n\rightarrow\infty}{\longrightarrow}}\gamma\Longleftrightarrow\frac{N_{t}}{t}\overset{a.e}{\underset{t\rightarrow\infty}{\longrightarrow}}\frac{1}{\gamma}
\]
valid even in the case when the sequence of inter-arrivals $T_{n}-T_{n-1}$
is ergodic. Both theorems (\ref{Thm}) and (\ref{Thm-1}) can be seen
as the continuous version of the main result stated in \cite{mgrw}.
However, the discrete version can follow easily from the continuous
case. That is, let $\left(x_{n}\right)$ be a sequence of integers
mimicking a simple walk on $\mathbb{Z}$ and let $r_{n}$ be the range
of $\left(x_{n}\right)$, i.e the number of explored sites up to time
$n$ by $\left(x_{n}\right)$. Let $x_{t}^{\star}$ be the function
obtained from $\left(x_{n}\right)$ by connecting the dots. It is
clear that 
\[
\vert R_{t}(x^{\star})-r_{\lfloor t\rfloor}\vert\leq2.
\]
Hence, in case of existence of limits, we obtain 
\[
\lim_{n\rightarrow\infty}\frac{r_{n}}{n}=\lim_{\lfloor t\rfloor\rightarrow\infty}\frac{r_{\lfloor t\rfloor}}{\lfloor t\rfloor}=\lim_{t\rightarrow\infty}\frac{R_{t}(x^{\star})}{t}.
\]

A typical process derived from Brownian motion is the so called Bessel process.
Bessel processes are important as they interfere in financial modeling
of stock market prices etc. A $n$-dimensional Bessel process, denoted
by $BES(n)$, is defined by 
\[
BES(n)=\Vert\boldsymbol{B}_{t}\Vert
\]
where $\Vert\cdot\Vert$ is the Euclidean norm and $\mathbf{B}_{t}$
is a $n$-dimensional Brownian motion \footnote{People often add the starting point as a subscript.}.
Another alternative definition of Bessel processes uses SDE theory.
More precisely, a Bessel process is the solution of the following SDE
\[
dX_{t}=dB_{t}+\frac{n-1}{2}\frac{dt}{X_{t}}.
\]
We refer the reader to \cite{lawler2018notes} for more details about the
subject. The determination of the statistic of the range of a Bessel
process is not an easy task, not to mention 
its long run behavior.
Now, theorems (\ref{Thm}) and (\ref{Thm-1}) allows us to obtain
the asymptotic behavior of such a range by knowing only the behavior
of the process itself. However, we provide a result that covers more
than the case of Bessel processes. 
\begin{prop}
Let $\boldsymbol{B}_{t}=(\boldsymbol{B}_{t}^{(1)},\boldsymbol{B}_{t}^{(2)},...,\boldsymbol{B}_{t}^{(n)})$
be an $n$-dimensional Brownian motion and let $0<p\leq+\infty$. Then
\[
\frac{R_{t}(\Vert\boldsymbol{B}\Vert_{p})}{t}\overset{a.e}{\underset{t\rightarrow\infty}{\longrightarrow}}0\,\,\,\,\,\,(a.e)
\]
where $\Vert\boldsymbol{B}\Vert_{p}$ is the $p^{th}$ norm \footnote{ It is not a true norm when $p<1$. We use the word by abuse of of terminology.} of $\boldsymbol{B}_{t}$,
i.e
\[
\Vert\boldsymbol{B}\Vert_{p}:=\begin{cases}
\left(\vert B_{t}^{(1)}\vert^{p}+\vert B_{t}^{(2)}\vert^{p}+...+\vert B_{t}^{(n)}\vert^{p}\right)^{\frac{1}{p}} & p<+\infty\\
\max_{1\leq i\leq n}\vert\boldsymbol{B}_{t}^{(i)}\vert & p=+\infty.
\end{cases}
\]
 
\end{prop}

\begin{proof}
Let $p<+\infty$. The law of large number implies that 
\[
\frac{\vert B_{t}^{(j)}\vert^{p}}{t^{p}}\overset{a.e}{\underset{t\rightarrow\infty}{\longrightarrow}}0
\]
for $j=1,...,n$. It follows that 
\[
\begin{alignedat}{1}\frac{\Vert\boldsymbol{B}\Vert_{p}}{t} & =\left(\frac{\vert B_{t}^{(1)}\vert^{p}}{t^{p}}+\frac{\vert B_{t}^{(2)}\vert^{p}}{t^{p}}+...+\frac{\vert B_{t}^{(n)}\vert^{p}}{t^{p}}\right)^{\frac{1}{p}}\\
 & \overset{a.e}{\underset{t\rightarrow\infty}{\longrightarrow}}0
\end{alignedat}
.
\]
For $p=+\infty$ we have 
\[
\begin{alignedat}{1}\frac{\Vert\boldsymbol{B}\Vert_{p}}{t} & \leq\left(\frac{\vert B_{t}^{(1)}\vert}{t}+\frac{\vert B_{t}^{(2)}\vert}{t}+...+\frac{\vert B_{t}^{(n)}\vert}{t}\right)\\
 & \overset{a.e}{\underset{t\rightarrow\infty}{\longrightarrow}}0
\end{alignedat}
.
\]
Therefore, in both cases 
\[
\frac{R_{t}(\Vert\boldsymbol{B}\Vert_{p})}{t}\overset{a.e}{\underset{t\rightarrow\infty}{\longrightarrow}}0.
\]
\end{proof}
\begin{cor}
The range of a Bessel process is negligible against the time. 
\end{cor}

The following result illustrates the relation between a deterministic
function $f$ and its underlying supremum. 
\begin{prop}
\label{prop:10} Let $f,\psi:[0,\infty)\rightarrow\mathbb{R}$ be
two functions such that $\psi(t)$ is positively increasing to infinity
and let $\ell>0$. If 
\begin{equation}
\frac{f(t)}{\psi(t)}\underset{t\rightarrow\infty}{\longrightarrow}\ell\label{f/psi-2}
\end{equation}
then 
\[
\frac{\sup_{0\leq s\leq t}f_{s}}{f(t)}\underset{t\rightarrow\infty}{\longrightarrow}1.
\]
\end{prop}

\begin{proof}
From the assumption, we know for every $\varepsilon>0$ there exists
a large enough $T$ such that, when $t>T$ 
\begin{equation}
\psi(t)(\ell-\varepsilon)\leq f(t)\leq\psi(t)(\ell+\varepsilon).\label{R_t-2}
\end{equation}
Also, 
\begin{equation}
1=\frac{f(t)}{f(t)}\leq\frac{\sup_{0\leq s\leq t}f_{s}}{f(t)}
\end{equation}
for $t>T$. On the other hand, we have
\begin{equation}
\frac{\sup_{0\leq s\leq t}f_{s}}{f(t)}\leq\frac{\sup_{0\leq s\leq T}f_{s}+\sup_{T\leq s\leq t}f_{s}}{f(t)}\leq\frac{\sup_{0\leq s\leq T}f_{s}+\psi(t)(\ell+\varepsilon)}{f(t)}
\end{equation}
Letting $t$ go to infinity, $\frac{\sup_{0\leq s\leq T}f_{s}}{f(t)}$
will converge to $0$ as $f(t)$ diverges to $+\infty$. $\frac{\psi(t)(\ell+\varepsilon)}{f(t)}$
will converge to $1+\frac{\varepsilon}{\ell}$. Therefore, 
\begin{equation}
1\leq\liminf_{t\to\infty}\frac{\sup_{0\leq s\leq t}f_{s}}{f(t)}\leq\limsup_{t\to\infty}\frac{\sup_{0\leq s\leq t}f_{s}}{f(t)}\leq1+\frac{\varepsilon}{\ell}
\end{equation}
The inequality (19) holds for any $\varepsilon>0$. Thus, 
\begin{equation}
\frac{\sup_{0\leq s\leq t}f_{s}}{f(t)}\underset{t\rightarrow\infty}{\longrightarrow}1.
\end{equation}
\end{proof}
\begin{cor}
Under the same assumptions of Proposition \eqref{prop:10}, we have
\[
\frac{\sup_{0\leq s\leq t}f_{s}}{\psi(t)}\underset{t\rightarrow\infty}{\longrightarrow}\ell
\]
and 
\[
\lim_{t\rightarrow+\infty}\frac{R_{t}(f)}{\psi(t)}=\lim_{t\rightarrow+\infty}\frac{R_{t}(\sup f)}{\psi(t)}=\ell.
\]
\end{cor}

In other words, the four functions $f,\sup_{0\leq s\leq t}f_{s},R_{t}(f),R_{t}(\sup f)$
have the same long run behavior under the assumptions of proposition
\eqref{prop:10}. Note that proposition \eqref{prop:10} is still
valid when $\ell=0$ provided that $f$ becomes eventually positive after certain
time. The following result is a consequence of proposition \eqref{prop:10}. It concerns the underlying supremum process of a drifted Brownian motion.
\begin{prop}\label{sup}
Let $\widetilde{V}_{t}^{\eta}=\sup_{0\leq s\leq t}V_{s}^{\eta}$. Then  
\end{prop}

\[
\frac{\widetilde{V}_{t}^{\eta}}{t}\overset{a.e}{\underset{t\rightarrow\infty}{\longrightarrow}}\eta
\]
and 
\[
\frac{R_{t}(\widetilde{V}^{\eta})}{t}\overset{a.e}{\underset{t\rightarrow\infty}{\longrightarrow}}\eta.
\]
\\
\\
A corollary of proposition (\ref{sup}) is that the supremum of a Brownian motion, as well as its range are negligible against the time evolution. However, it would be interesting to derive estimates for the range like the law of iterated logarithm for Brownian motion, and on top of that investigate whether such potential properties are deterministic or not.

\subsection*{Acknowledgments}
The authors would like to thank Greg Markowsky for helpful communications.

\bibliographystyle{plain}
\bibliography{Maher}

\end{document}